%% file: main.tex
\documentclass[letterpaper, 10 pt, conference]{ieeeconf}  

\IEEEoverridecommandlockouts                              
\usepackage{graphics} 
\usepackage{epsfig} 
\usepackage{times} 

\usepackage{amsmath,amssymb,amsthm} 

\usepackage{tikz}

\usepackage{optidef}

\usepackage{url}

\newcommand{\norm}[1]{\|#1\|}

\newcommand{\bmtx}{\begin{bmatrix}}
    \newcommand{\emtx}{\end{bmatrix}}
\newcommand{\bsmtx}{\left[ \begin{smallmatrix}} 
    \newcommand{\esmtx}{\end{smallmatrix} \right]}

\newcommand{\LMIVarXW}{\bmtx x \\ w \emtx}
\newcommand{\sLMIVarXW}{\bsmtx x \\ w \esmtx}
\newcommand{\Ealpha}{\mathcal{E}_\alpha}
\newcommand{\Rn}{\mathbb{R}^{n}}
\newcommand{\Ri}[1]{\mathbb{R}^{#1}}

\newcommand{\commentg}[1]{\textcolor{black}{#1}}

\newtheorem{theorem}{Theorem}

\newtheorem{lemma}{Lemma}

\title{\LARGE \bf  Quadratic Constraints for Local Stability Analysis of Quadratic Systems}

\author{Shih-Chi Liao$^{1}$, Maziar S. Hemati$^{2}$, and Peter Seiler$^{1}$
    \thanks{*This research was sponsored by the US Army Research Office and was accomplished under Grant Number W911NF-20-1-0156. The work of Maziar S. Hemati was supported in part by the Air Force Office of Scientific Research under award numbers FA9550-21-1-0106 and FA9550-21-1-0434, the National Science Foundation under award number CBET-1943988, and the Office of Naval Research under award number N000140-22-1-2029.}
    \thanks{$^{1}$Shih-Chi Liao and Peter Seiler are with the Department of Electrical Engineering and Computer Sciences at the University of Michigan, Ann Arbor, {\tt\small \{shihchil,pseiler\}@umich.edu}}%
    \thanks{$^{2}$Maziar S. Hemati is with the Department of Aerospace Engineering and Mechanics at the University of Minnesota, {\tt\small mhemati@umn.edu}}
}

\begin{document}

\maketitle
\thispagestyle{empty}
\pagestyle{empty}

\begin{abstract}





This paper proposes new quadratic constraints (QCs) to bound a quadratic polynomial. Such QCs can be used in dissipation ineqaulities to analyze the stability and performance of nonlinear systems with quadratic vector fields. The proposed QCs utilize the sign-indefiniteness of certain classes of quadratic polynomials. These new QCs provide a tight bound on the quadratic terms along specific directions. This reduces the conservatism of the QC bounds as compared to the QCs in previous work. Two numerical examples of local stability analysis are provided to demonstrate the effectiveness of the proposed QCs.

\end{abstract}



\input{Content/1_Introduction}

\input{Content/2_Problem}

\input{Content/3_Methods}

\input{Content/4_Results}

\input{Content/5_Conclusions}

\input{Content/Appendix}

\addtolength{\textheight}{-12cm}   


\section*{ACKNOWLEDGMENT}
The authors would like to thank Talha Mushtaq and Diganta Bhattacharjee for valuable discussion.


\bibliographystyle{ieeetr}
\bibliography{Reference}


\end{document}

%% file: Content/1_Introduction.tex
\section{INTRODUCTION}



Quadratic systems are an important class of nonlinear dynamics. A generic nonlinear system can be approximated by a quadratic system through a Taylor series expansion. This improves the approximation compared to linearization~\cite{Khalil:1173048}. Further, some systems are directly modeled by quadratic dynamics: e.g.,~fluid flows governed by the incompressible Navier-Stokes equations. These dynamics are quadratic, and linear analysis is often insufficient due to significant nonlinear effects~\cite{schmid2000stability}. Furthermore, quadratic systems can model complex nonlinear behavior such as chaos~\cite{lorenz1963deterministic} and limit cycle oscillations~\cite{kuznetsov2013visualization}. Thus, approaches to analyze quadratic systems can benefit scientific and engineering applications.

Dissipation inequalities can be used to analyze many dynamical system properties, such as stability, reachability, and robustness~\cite{arcak2016networks}. The analysis approach generally involves searching for a valid storage function that certifies the dissipativity. The certification can often be posed as a convex optimization problem, such as a semi-definite program (SDP). These convex optimization problems can be solved efficiently, enabling system analysis and control design algorithms~\cite{boyd1994linear}. 

Quadratic constraint (QC) is a modeling framework that abstracts a nonlinearity as a quadratic inequality of the input and output of the functions~\cite{megretski1997system}. QCs allow one to analyze nonlinear systems through dissipation inequalities~\cite{seiler2014stabilityIQC} at the expense of conservatism due to abstraction. For quadratic polynomials, a few local QCs are proposed in the literature in the context of region of attraction (ROA) analysis for fluid systems. QCs were derived in~\cite{kalur2021PRF} and~\cite{Liu2020io-inspired} to bound a quadratic polynomial in a spherical local region. These QCs were further generalized to an ellipsoidal local region in~\cite{kalur2021LCSS}. Recently,~\cite{toso2022regional} proposed QCs to capture the interaction of quadratic polynomials in a hyperrectangle. These works pursued QC-based approaches over prevailing sum-of-squares optimization techniques~\cite{parrilo2000structured} in order to achieve scalable algorithms for large-dimensional systems.

In this paper, we explore the function landscape of quadratic polynomials and proposed new QCs to tighten the description along the direction which the function equals to zero. 
These QCs reduce the conservatism of QCs presented in~\cite{kalur2021LCSS}, and can also generalize the QCs proposed in~\cite{toso2022regional}. 
%
Finally, the effectiveness of the proposed QCs are investigated with two numerical example with ROA estimation problems.


%% file: Content/2_Problem.tex
\section{Problem Formulation}


\input{ProblemFormulation/NLSystem}

\input{ProblemFormulation/ExistingQCs}

\input{ProblemFormulation/ROA_Lyapunov}

\input{ProblemFormulation/ConservatismOfExisitingQC}

%% file: ProblemFormulation/NLSystem.tex
\subsection{Quadratic Nonlinear System}

There are $m=\frac{n^2+n}{2}$ quadratic monomials that can be constructed from $x\in\Rn$. Let $z: \Rn \rightarrow\Ri{m}$ denote the function that constructs the vector of such monomials:
\begin{align}
    z(x) = \bmtx x_1^2 & x_1x_2 & \dots & x_2^2 & x_2x_3 & \dots & x_n^2 \emtx^\top,     \label{eq:QuadMon}
\end{align}

Note that any homogeneous quadratic function $\phi : \Rn\rightarrow\Ri{}$ is a linear combination of quadratic monomials. In other words, if $\phi(x) = x^\top Q x$ for some matrix $Q = Q^\top \in \Ri{n\times n}$ then there exists $b\in\Ri{m}$ such that $\phi(x) = b^\top z(x)$. Note that the matrix $Q$ can be constructed from the Hessian of $\phi$: $Q = \frac{1}{2}\nabla^2\phi$. Both forms for a quadratic function (expressed as $x^\top Q x$ or $b^\top z(x)$) will be used throughout the paper.

Consider a quadratic polynomial system of the form:
\begin{align} 
    \dot{x}(t) = A x(t) + B z(x(t)),    \label{eq:NLsys}
\end{align}
where $A\in\Ri{n\times n}$ and $B\in\Ri{n\times m}$. 
The system can have multiple equilibrium points in general but we focus on $x_e=0$. We assume $A$ is Hurwitz so that $x_e=0$ is locally asymptotically stable. Other equilibrium points can be shifted to the origin via a coordination transformation to get the same form of quadratic system~\eqref{eq:NLsys} as shown in~\cite{amato2006region}.

The Lur'e decomposition~\cite{Khalil:1173048} poses the system~\eqref{eq:NLsys} as:
\begin{align}
    \begin{split}
        \dot{x}(t) &= Ax(t) + Bw(t) \\
        w(t) &= z(x(t)).
    \end{split}     \label{eq:LureSys}
\end{align}
The Lur'e decomposition separates the linear time-invariant dynamics from the quadratic nonlinearity as shown in Fig.~\ref{fig:Lure}. This decomposition enables one to analyze the quadratic system using dissipation inequality with QCs~\cite{seiler2014stabilityIQC}.

\input{Diagram/Lure}

%% file: Diagram/Lure.tex
\begin{figure}[h]
    \centering
    \vspace*{0.3cm}
    \begin{picture}(120,70)(20,30)
        \thicklines
        \put(40,25){\framebox(90,40){ {\Large $\dot{x}=Ax+Bw$} }}
        \put(70,75){\framebox(30,30){ {\Large $z(\cdot)$} }}
        \put(-3,65){ {\LARGE $x$ }}
        \put(15,45){\line(1,0){25}}  
        \put(15,45){\line(0,1){45}}  
        \put(15,90){\vector(1,0){55}}  
        \put(158,65){ {\LARGE $w$ }}
        \put(155,90){\line(-1,0){55}}  
        \put(155,45){\line(0,1){45}}  
        \put(155,45){\vector(-1,0){25}}  
    \end{picture}
    \caption{Lur'e decomposition of quadratic system~\eqref{eq:NLsys}}
    \label{fig:Lure}
\end{figure}
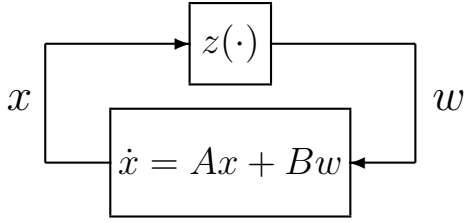

%% file: ProblemFormulation/ExistingQCs.tex
\subsection{Existing Local Quadratic Constraints}

The effect of the nonlinearity $z(x)$ can be bounded in a local region $\mathcal{D}\subset\Rn$ using quadratic constraints (QCs). These QCs take the following form:
\begin{align}
    \LMIVarXW^\top M_i \LMIVarXW \geq 0 & & \forall w=z(x), x \in\mathcal{D},     \label{eq:QC}
\end{align}
where $M_i\in\Ri{(n+m) \times (n+m)}$ and $i=1.\dots,k$.
Note that the subscript $i$ is an indexing number, since a nonlinearity can satisfy multiple QCs. Throughout the paper, we consider a local ellipsoidal region of the form $\mathcal{D}=\Ealpha := \{x : x^\top E x \leq \alpha^2\}$, where $E = E^\top \in\Ri{n\times n}$ is a positive definite matrix and $\alpha\in\mathbb{R}$ is a positive scalar.

We summarize two types of QCs that have been developed in the literature. These form the foundation of our new QCs presented in Section~\ref{sec:III_QCs}. The QC in Lemma 1 below is stated in~\cite{kalur2021LCSS} and generalizes results in~\cite{Liu2020io-inspired}.
\begin{lemma}[Cauchy–Schwarz QC]    \label{lemma:CSQC}
    Let the local ellipsoidal region $\Ealpha$ be given. A nonlinearity $\phi(x) = x^\top Q x$ satisfies the following local QC:
    \begin{align}
        \bmtx x \\ \phi(x) \emtx^\top 
        \bmtx \alpha^2(Q E^{-1} Q) & 0 \\ 0 & -1 \emtx 
        \bmtx x \\ \phi(x) \emtx \geq 0\,
        \forall x \in \Ealpha.      \label{eq:CSQC}
    \end{align}
\end{lemma}
The Cauchy-Schwarz QC (CSQC) is named as it involves Cauchy-Schwarz in the bounding process. This QC~\eqref{eq:CSQC} can be re-written in the form shown in~\eqref{eq:QC}.  Specifically, express the quadratic nonlinearity in the form $\phi(x) = b^\top w$ and substitute $\bsmtx x \\ \phi(x) \esmtx = \bsmtx I & 0 \\ 0 & b^\top \esmtx \sLMIVarXW$ into~\eqref{eq:CSQC}. Lemma~\ref{lemma:CSQC} provides a constraint for an arbitrary quadratic function on an ellipsoid. The next lemma provides a bound for products of quadratic functions with special structure. It was originally proposed in Section IV-A of~\cite{toso2022regional}. 
\begin{lemma}       \label{lemma:TosoQC}
    Let the local region $\Ealpha$ and two quadratic function $\phi_1(x)$ and $\phi_2(x)$ be given. If $\phi_1(x)\phi_2(x) = x_i^2x_jx_k$ with $j\neq k$, then the QC holds:
    \begin{align}
        \bsmtx x \\ \phi_1(x) \\ \phi_2(x) \esmtx^\top
        \bsmtx \alpha^2(E^{-1})_{ii}cc^\top & 0 & 0 \\ 0 & 0 & -1 \\ 0 & -1 & 0 \esmtx 
        \bsmtx x \\ \phi_1(x) \\ \phi_2(x) \esmtx  \geq 0\,  
        \forall x \in \Ealpha,       \label{eq:TosoQC}
    \end{align}
    where $(E^{-1})_{ii}$ is the $(i,i)$ entry of $E^{-1}$, $c = e_j + e_k$, and $e_j,e_k\in\Rn$ are standard basis vectors. 
\end{lemma}

A similar variable substitution can be used to re-write~\eqref{eq:TosoQC} in the form of~\eqref{eq:QC}. The QC in~\cite{toso2022regional} was formulated using a hyperrectangle for the local region. Lemma~\ref{lemma:TosoQC} is a variation stated using an ellipsoid $\Ealpha$ for the local region. This causes a slight difference in the coefficient matrix in the QC. Section~\ref{sec:CrossQC} will present a more general QC (with proof) which includes Lemma~\ref{lemma:TosoQC}.

%% file: ProblemFormulation/ROA_Lyapunov.tex
\subsection{Local Stability Condition with QC and Lyapunov Stability}

The QCs~\eqref{eq:QC} can be used to formulate a Lyapunov condition for local stability analysis. Here, we illustrate the approach with a condition to estimate the region of attraction (ROA) for the system in~\eqref{eq:NLsys}. The ROA of an equilibrium $x_e=0$ is defined as the set of initial conditions for which the solution $x(t)$ of~\eqref{eq:NLsys} asymptotically converges to the equilibrium. The next theorem from~\cite{kalur2021LCSS} provides a matrix inequality condition that gives a spherical ROA estimate of the system~\eqref{eq:NLsys}.
\begin{theorem}     \label{theorem:ROA}
    Let $E=E^\top\succ0$ and $\alpha>0$ be given. Moreover, assume the nonlinearity $z(\cdot)$ in the system~\eqref{eq:LureSys} satisfies a set of QCs~\eqref{eq:QC}. If $\exists P=P^\top\in\Ri{n\times n}, r>0$ and $\xi_1,\dots\xi_k\in\Ri{}$ such that:
    \begin{subequations}    \label{eq:ROA}
    \begin{align}
        &\bmtx A^\top P + PA & PB \\ B^\top P & 0 \emtx + \sum_{i=1}^k \xi_i M_i \prec 
        \bmtx -\epsilon I & 0 \\ 0 & 0 \emtx \label{eq:ROA_Vdot}\\
        &\frac{1}{\alpha^2}E \preceq P \preceq \frac{1}{r^2}I \label{eq:ROA_set}\\
        &\xi_i \geq 0 \mbox{ for } i=1,\dots,k, \label{eq:ROA_LagMul}
    \end{align}
    \end{subequations}
    then $x_e=0$ is a locally asymptotically stable equilibrium. Moreover, $\{x:x^\top x \leq r^2\}$ is a ROA estimate of system~\eqref{eq:NLsys}.
\end{theorem}
\begin{proof}
The proof relies on standard Lyapunov stability arguments~\cite{Khalil:1173048} combined with QCs~\cite{megretski1997system}. A proof is given in~\cite{kalur2021LCSS} and~\cite{Liu2020io-inspired} but is briefly summarized here for completeness. Define the Lyapunov function $V(x) := x^\top P x$. Inequality~\eqref{eq:ROA_set} implies that $V$ is positive definite. Left/right multiply~\eqref{eq:ROA_Vdot} by $\bmtx x(t)^\top & w(t)^\top \emtx$ and its transpose to show:
\begin{align*}
    \frac{d}{dt}V(x(t)) + \sum_{i=1}^k \xi_i \bmtx x(t) \\ w(t) \emtx^\top M_i \bmtx x(t) \\ w(t) \emtx < 
    -\epsilon \norm{x(t)}_2^2.
\end{align*}
This implies $\frac{d}{dt}V(x(t))< -\epsilon\norm{x(t)}_2^2$ for any $x(t)\in\Ealpha$ since $\xi_i$ and the QCs are non-negative for any $x(t)\in\Ealpha$. The equilibrium $x_e=0$ is locally asymptotically stable by Lyapunov stability theory~\cite{Khalil:1173048}. 

Inequality~\eqref{eq:ROA_set} implies that the level set $\{x:x^\top P x \leq 1\}$ is contained in $\Ealpha$. Thus $x(t)$ converges to $x_e$ for any initial condition in the level set $\{x:x^\top P x\leq 1 \}$, i.e. the level-set is contained in the region of attraction. Finally, inequality~\eqref{eq:ROA_set} implies that the spherical set $\{x:x^\top x \leq r^2\}$ is contained in the level set $\{x:x^\top P x\leq 1\}$. 
\end{proof}


Note that Lyapunov stability condition can be viewed as a special case of dissipation inequality. Similar conditions as in Theorem~\ref{theorem:ROA} can be formulated for other system properties, such as reachability, robustness, and performance.

%% file: ProblemFormulation/ConservatismOfExisitingQC.tex
\subsection{Conservatism of Existing QCs}



The QCs bound the effect of the nonlinearity in the local region. This enables the estimate of the ROA of~\eqref{eq:NLsys} (or other system properties) via Lyapunov or dissipation inequality conditions. However, if the QC bounds the nonlinearity too "loosely" then the analysis condition will be conservative. The remainder of this section provides an example to illustrate this issue. This motivates the construction of new QCs in Section~\ref{sec:III_QCs}.

Here, we present that the CSQC~\eqref{eq:CSQC} fails to tightly bound a quadratic function $x^\top Q x$ where $Q$ is sign-indefinite. To illustrate, consider the case $\phi(x) = x_1x_2$ in the local region of a unit sphere ($E=I$ and $\alpha = 1$). The function $\phi(x)$ corresponds to the matrix $Q = \bsmtx 0 & 0.5 \\ 0.5 & 0 \esmtx$. The CSQC~\eqref{eq:CSQC} on $\phi(x)$ corresponds to the inequality:
\begin{align}
    0.25(x_1^2+x_2^2) \geq \phi(x)^2 & & \forall x^\top x\leq 1.  \label{eq:CSQC_phi}
\end{align}
Fig.~\ref{fig:CSQC_xixj} visualizes each side of the inequality~\eqref{eq:CSQC_phi}. The landscape of the right-hand side ($\phi(x)^2$, green surface) has peaks along the directions $x_1=\pm x_2$ and valleys along $\phi(x) = 0$ ($x_1=0$ and $x_2=0$). Note that the left side ($0.25(x_1^2 + x_2^2)$, blue surface) provides a tight upper bound of the green surface along the peaks. However, the blue surface provides a loose bound along the valleys of the green surface.

These landscape properties (peaks and valleys) are an inherent feature of a sign-indefinite quadratic function $\phi(x)$, as they have multiple directions for which $\phi(x) = 0$. Hence, the loose bound of QCSC~\eqref{eq:CSQC} does not depend on the shape and the size of the ellipsoidal region $\Ealpha$. Furthermore, this looseness will potentially lead to conservative analysis. The next section proposes new QCs to reduce the conservatism.


\begin{figure}[hb]
    \centering
    \includegraphics[width=0.95\linewidth]{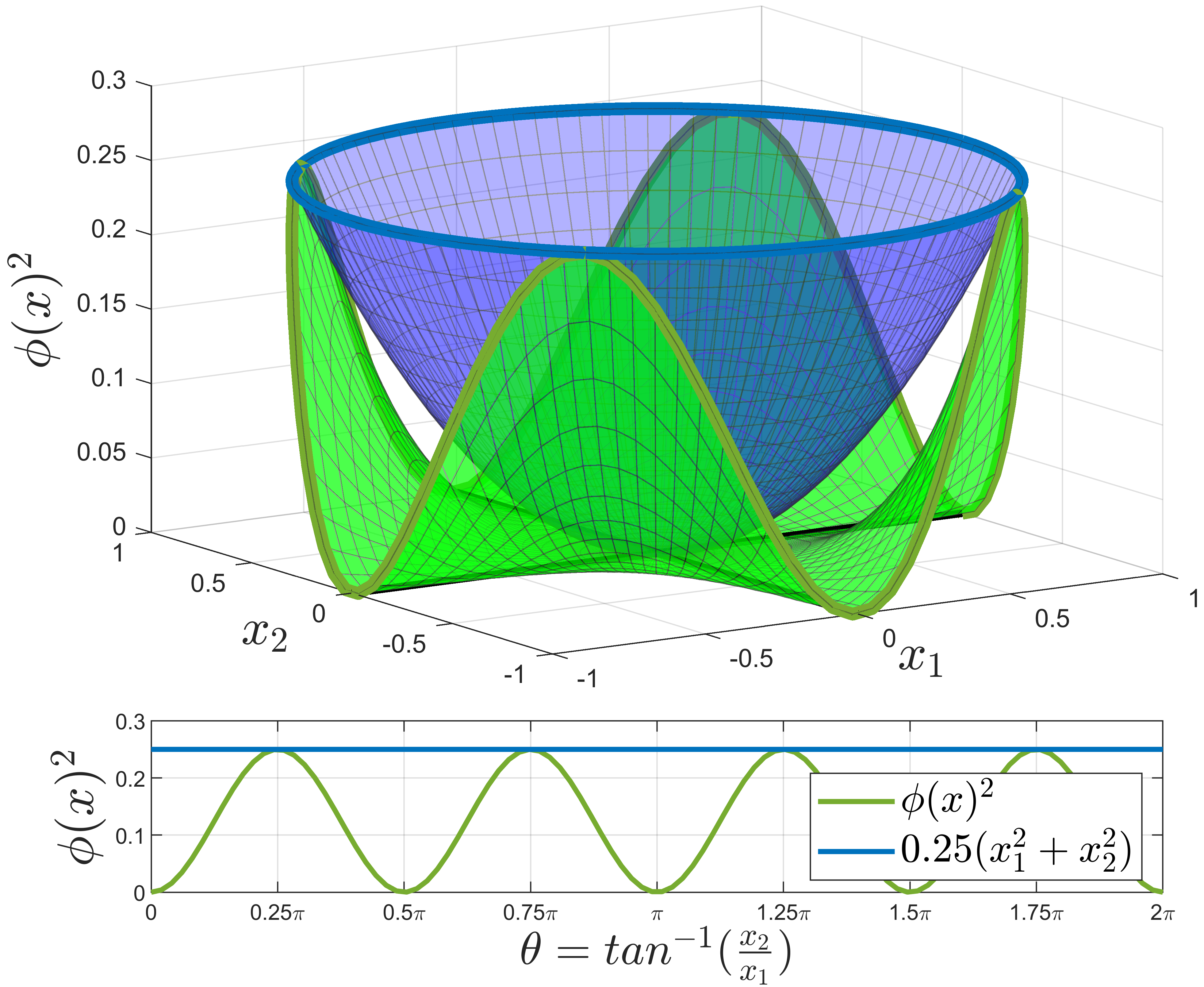}
    \caption{ Visualization of each side of the QC~\eqref{eq:CSQC_phi}. The top subplot shows the right-hand side ($\phi(x)^2$) and the left-hand side ($0.25(x_1^2+x_2^2)$ in a unit sphere. The lower subplot shows the two sides of QC on the boundary of the unit sphere $(x^\top x=1)$. The blue surface tightly bounds the green only at the peaks, but provides a loose bound at the valleys.}
    \label{fig:CSQC_xixj}
\end{figure}

%% file: Content/3_Methods.tex
\section{Local QCs on Quadratic Nonlinearities}     \label{sec:III_QCs}


New QCs are introduced in Section~\ref{sec:ValleyQC} and~\ref{sec:Rank3QC} to reduce the conservatism of CSQC~\eqref{eq:CSQC} by capturing the landscape properties of sign-indefinite quadratic functions. Specifically, QCs are presented for sign-indefinite quadratic function $x^\top Q x$ with $Q$ being rank-2 and rank-3. Furthermore, the method derived in Section~\ref{sec:ValleyQC} is applied to generalize the QC~\eqref{eq:TosoQC} in Section~\ref{sec:CrossQC}.


\input{Methods/ValleyQC_mon}

\input{Methods/ValleyQC_rank3}

\input{Methods/TOSO_QC}


%% file: Methods/ValleyQC_mon.tex
\subsection{QCs on Rank-2 Sign-indefinite Quadratic Functions}     \label{sec:ValleyQC}


The CSQC in~\eqref{eq:CSQC_phi} provides a bound for $\phi(x) = x_1x_2$ on the unit sphere. Consider the following alternative bound:
\begin{align}
    x_1^2 \geq \phi(x)^2 & & \forall x^\top x \leq 1.    \label{eq:ValleyQC_xi}
\end{align}
This is a valid QC since $x_1^2-\phi(x)^2 = (1-x_2^2)x_1^2$ and $x_2^2\leq x^\top x \leq 1$. Fig.~\ref{fig:ValleyQC} illustrates that the left side of~\eqref{eq:ValleyQC_xi} ($x_1^2$, red surface) provides an upper bound of the right side ($\phi(x)^2$, green surface). Furthermore, this QC is specifically tight along the direction $x_1 = 0$. Similarly, the inequality 
\begin{align}
    x_2^2 \geq \phi(x)^2 & &  \forall x^\top x \leq 1     \label{eq:ValleyQC_xj}
\end{align}
is also a valid QC in the unit sphere. The left side of QC~\eqref{eq:ValleyQC_xj} corresponds to a similar surface as the red surface with 90-degree rotation, i.e., it tightly bounds $\phi(x)^2$ along the direction $x_2=0$. Note that the two QCs~\eqref{eq:ValleyQC_xi} and~\eqref{eq:ValleyQC_xj} are each tight on one valley of $\phi(x)^2$. QCs~\eqref{eq:ValleyQC_xi} and~\eqref{eq:ValleyQC_xj} together with the CSQC~\eqref{eq:CSQC_phi} tightly bound the peaks and valleys of the quadratic function $\phi(x) = x_1x_2$.

\begin{figure}[b!]
    \centering
    \includegraphics[width=0.95\linewidth]{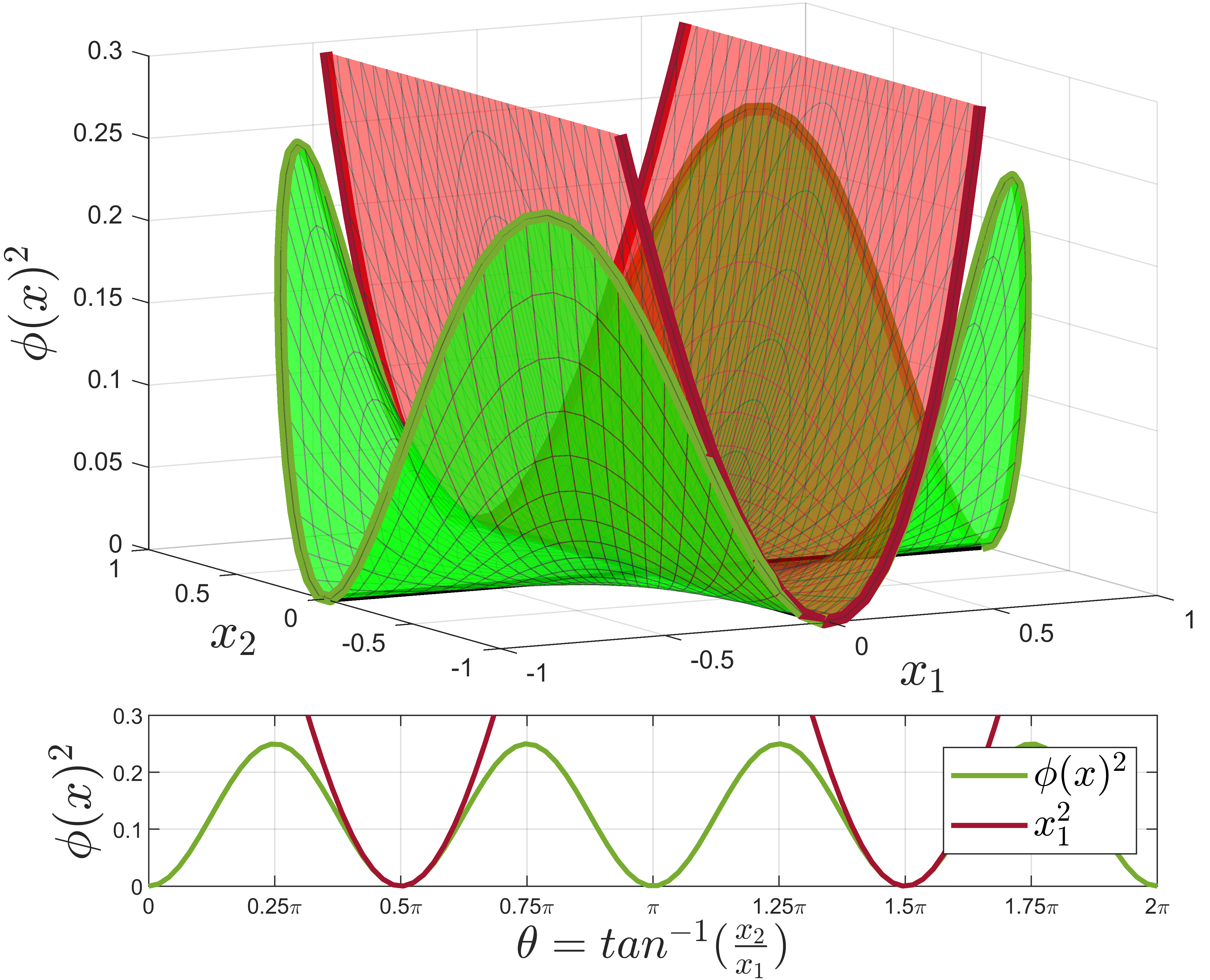}
    \caption{Visualization of each side of the QC~\eqref{eq:ValleyQC_xi}. The top subplot shows the right-hand side ($\phi(x)^2$) and the left-hand side ($x_1^2$) in on a unit sphere. The lower subplot shows the two sides of the QC on the boundary of the unit sphere $(x^\top x=1)$. The red surface tightly bounds the green along one of the valleys, but provides a loose bound at the peaks and the other valley.}
    \label{fig:ValleyQC}
\end{figure}

Here, we generalize the above QCs beyond quadratic monomials $\phi(x) = x_1x_2$ to any quadratic function having a similar function landscape. The following lemmas establish the proposed constraints using the largest eigenvalue of $Q$, denoted as $\lambda_{max}(Q)$.
\begin{lemma}       \label{lemma:maxEig}
    Let $E=E^\top\succ0, \alpha>0$, and matrix $Q=Q^\top\in\Ri{n\times n}$ be given. Assume $\lambda_{max}(Q) > 0$. Then:
    \begin{align*}
        \alpha^2\lambda_{max} (\tilde{Q}) = 
        &\max_{x\in\Ealpha} \quad x^\top Q x,
    \end{align*}
    where $\tilde{Q}=E^{-\frac{1}{2}}QE^{-\frac{1}{2}}$. Furthermore, if $Q = cc^\top$ for some nonzero $c\in\Rn$, then  $\lambda_{max}(\tilde{Q}) = c^\top E^{-1}c > 0$.
\end{lemma}
\begin{proof}
    Define $y = \frac{1}{\alpha}E^{\frac{1}{2}}x$ and $\Tilde{Q} = E^{-\frac{1}{2}}QE^{-\frac{1}{2}}$ so that the constrained optimization problem becomes:
    \begin{align*}
        \max_{y^\top y \leq 1} \alpha^2 y^\top \tilde{Q} y.
    \end{align*}
    Note that $\lambda_{max}(Q)>0$ implies $\lambda_{max}(\tilde{Q}) > 0$ (Theorem 4.5.8~\cite{horn2012matrix}). Hence, the problem corresponds to finding the largest eigenvalue of $\tilde{Q}$ (Theorem 4.2.2~\cite{horn2012matrix}). Furthermore, if $Q = cc^\top$ is an outer product of a vector $c\in\Rn$, then $\tilde{Q}$ is rank-1 and has an eigenvector $E^{-\frac{1}{2}}c$ with associated eigenvalue $c^\top E^{-1}c > 0$.
\end{proof}

\begin{lemma}       \label{lemma:Valley}
    Let $E=E^\top\succ0, \alpha>0$ and vectors $c_1,c_2\in\Rn$ be given. The quadratic function $\phi(x) = x^\top Q x$ with $Q = \frac{1}{2}(c_1c_2^\top+c_2c_1^\top)$ satisfies the following inequality:
    \begin{align}   
        \bmtx x \\ \phi(x) \emtx^\top  
        \bmtx \alpha^2 W & 0 \\ 0 & -1 \emtx
        \bmtx x \\ \phi(x) \emtx    \geq 0 & & \forall x\in\Ealpha  \label{eq:lemma_Valley}
    \end{align}
    with $W= (c_1^\top E^{-1}c_1)c_2c_2^\top$ or $(c_2^\top E^{-1}c_2)c_1c_1^\top$. 
\end{lemma}
\begin{proof}
    Note that $\phi(x) = (c_1^\top x) (c_2^\top x)$. The following inequality holds for the case $W=(c_1^\top E^{-1} c_1) c_2c_2^\top$:
            \small
    \begin{align}
        \alpha^2x^\top W x - \phi(x)^2 &= (\alpha^2c_1^\top E^{-1}c_1 - (c_1^\top x)^2) (c_2^\top x)^2. \label{eq:lemma_Valley_proof}
    \end{align}
    \normalsize
    Lemma~\ref{lemma:maxEig} implies that $(c_1^\top x)^2 = x^\top (c_1c_1^\top)x \leq \alpha^2 c_1^\top E^{-1} c_1$ for any $x\in\Ealpha$. Thus,~\eqref{eq:lemma_Valley} follows from this inequality applied to~\eqref{eq:lemma_Valley_proof}. The case $W=(c_2^\top E^{-1}c_2)c_1c_1^\top$ is shown similarly.
\end{proof}


Utilizing the above lemmas, the next theorem provides QCs for any quadratic function $x^\top Qx$ with $Q$ having two non-zero eigenvalues of opposite sign. In particular, consider the case where $Q$ is rank 2 with one positive and one negative eigenvalue. Specifically, let $(\lambda_p, v_p)$ and $(\lambda_n, v_n)$ be the eigenpairs of $Q$ associated with the positive and negative eigenvalues. Then $Q$ can be written as $\frac{1}{2}(c_1c_2^\top+c_2c_1^\top)$ with $c_1 = \sqrt{\lambda_p}v_p + \sqrt{|\lambda_n|}v_n$ and $c_2 = \sqrt{\lambda_p}v_p - \sqrt{|\lambda_n|}v_n$. Note that $c_1$ and $c_2$ are nonzero, linearly independent vectors. These linear algebra facts are shown in Appendix~\ref{apdx:Rank2}.
The theorem below provides QCs for matrices of this form.

\begin{theorem}[Rank-2 Valley QC]       \label{theorem:ValleyQC}
    Let $c_1,c_2\in\Rn$ be two nonzero, linearly independent vectors. Define the quadratic function $\phi:\Rn\to\Ri{}$ as $\phi(x) = x^\top Q x$ where $Q=\frac{1}{2}(c_1c_2^\top+c_2c_1^\top)$. There exists $b\in\Ri{m}$ such that $\phi(x) = b^\top w$. Moreover, $\phi$ satisfies the following local QCs:
    \begin{align}   
        &\LMIVarXW^\top 
        \bmtx \alpha^2W & 0 \\ 0 & -bb^\top \emtx 
        \LMIVarXW \geq 0 \qquad \forall x\in\mathcal{E}_\alpha \label{eq:ValleyQC}
    \end{align}
    with $W = (c_2^\top E^{-1} c_2) c_1c_1^\top \text{ or } (c_1^\top E^{-1} c_1) c_2c_2^\top$.
\end{theorem}

\begin{proof}
    Note that $\phi(x) = (c_1^\top x)(c_2^\top x)$ and hence it follows from Lemma~\ref{lemma:Valley} that $\phi(x)^2 \leq \alpha^2 x^\top Wx$ for all $x$ in $\Ealpha$. Substitute $\phi(x) = b^\top w$ to obtain~\eqref{eq:ValleyQC}. 
\end{proof}

Note that the class of quadratic functions in Theorem~\ref{theorem:ValleyQC} is exactly the class with $Q$ being rank 2 with one positive and one negative eigenvalue (Appendix~\ref{apdx:Rank2}). Theorem~\ref{theorem:ValleyQC} provides two additional QCs along with CSQC~\eqref{eq:CSQC} that can capture this class of quadratic nonlinearity well in the local region $\Ealpha$. Inequalities~\eqref{eq:ValleyQC_xi} and~\eqref{eq:ValleyQC_xj} are examples of Rank-2 Valley QCs~\eqref{eq:ValleyQC} with $E=I, \alpha=1$, and $Q=\bsmtx 0 & 1\\ 1 & 0 \esmtx$.

%% file: Methods/ValleyQC_rank3.tex
\subsection{QCs on Rank-3 Sign-indefinite Quadratic Functions}        \label{sec:Rank3QC}

The concept of Rank-2 Valley QCs~\eqref{eq:ValleyQC} can be extended to more general quadratic functions. Here, we consider a quadratic function $x^\top Qx$, where $Q$ is rank 3 with two positive and one negative eigenvalue. Specifically, let $(\lambda_i, v_i)$ be the eigenparis of $Q$ for $i=1,2,3$ with $\lambda_1,\lambda_2$ being positive and $\lambda_3$ being negative. Then $Q$ can be written as $\frac{1}{2}(c_1c_2^\top + c_2c_1^\top)+c_3c_3^\top$, where $c_1 = \sqrt{\lambda_1}v_1 + \sqrt{|\lambda_3|}v_3, c_2 = \sqrt{\lambda_1}v_1 - \sqrt{|\lambda_3|}v_3$ and $c_3 = \sqrt{\lambda_2}v_2$. Note that $c_1,c_2$ are nonzero, linearly independent vectors with $c_3$ orthogonal to $c_1$ and $c_2$. These facts are shown in Appendix~\ref{apdx:Rank3}. The next theorem provides QCs for nonlinearity of this form.

\begin{theorem}[Rank-3 Valley QC]     \label{theorem:Rank3QC}
    Let $c_1,c_2,c_3\in\Rn$ be three nonzero, linearly independent vectors with $c_3$ orthogonal to $c_1$ and $c_2$. Define the quadratic function $\phi:\Rn\to\Ri{}$ as $\phi(x) = x^\top Qx$ where $Q=\frac{1}{2}(c_1c_2^\top + c_2c_1^\top)+c_3c_3^\top$. There exists $b\in\Ri{m}$ such that $\phi(x)=b^\top w$. Moreover, $\phi$ satisfies the following local QCs:
    \begin{align}   
        \LMIVarXW^\top
        &\bmtx \alpha^2(W + \gamma c_3c_3^\top) & 0 \\ 0 & -bb^\top \emtx
        \LMIVarXW \geq 0\,\,\forall x\in\mathcal{E}_\alpha \label{eq:Rank3ValleyQC_phi}
    \end{align}
    with $W= (c_1^\top E^{-1}c_1)c_2c_2^\top$ or $(c_2^\top E^{-1}c_2)c_1c_1^\top$ and $\gamma = \lambda_{max}(E^{-\frac{1}{2}}(2Q-c_3c_3^\top)E^{-\frac{1}{2}})$.
\end{theorem}
\begin{proof}
    Note that $\phi(x) = (c_1^\top x)(c_2^\top x) + (c_3^\top x)^2$. Lemma~\ref{lemma:maxEig} and~\ref{lemma:Valley} imply that the inequality below holds for all $x\in\Ealpha$:
    \begin{align*}
        \phi(x)^2 &= (c_1^\top x)^2(c_2^\top x)^2 + (2(c_1^\top x)(c_2^\top x) + (c_3^\top x)^2) (c_3^\top x)^2 \\
        &= (c_1^\top x)^2(c_2^\top x)^2 + (x^\top(2Q - c_3c_3^\top)x) (c_3^\top x)^2 \\
        &\leq \alpha^2 x^\top W x + \alpha^2 \gamma x^\top c_3c_3^\top x    \quad \forall x \in\Ealpha.
    \end{align*}
    Substitute $\phi(x) = b^\top w$ and re-arrange the inequality into the quadratic form to obtain~\eqref{eq:Rank3ValleyQC_phi}
\end{proof}

Note that the class of quadratic functions in Theorem~\ref{theorem:Rank3QC} is exactly the class with $Q$ being rank 3 with two positive and one negative eigenvalue (Appendix~\ref{apdx:Rank3}). The Rank-3 Valley QCs~\eqref{eq:Rank3ValleyQC_phi} tighten the characterization of this class of quadratic function beside the CSQC~\eqref{eq:CSQC}. Furthermore, $Q$ can be alternatively written as $\frac{1}{2}(c_1c_2^\top + c_2c_1^\top)+c_3c_3^\top$, where $c_1 = \sqrt{\lambda_2}v_2 + \sqrt{|\lambda_3|}v_3, c_2 = \sqrt{\lambda_2}v_2 - \sqrt{|\lambda_3|}v_3$ and $c_3 = \sqrt{\lambda_1}v_1$. Hence, there exists four Rank-3 Valley QCs for a nonlinearity with rank-3 sign-indefinite matrix $Q$.

If a quadratic function $x^\top\tilde{Q}x$, where $\tilde{Q}$ has exactly one positive eigenvalue and exactly two negative eigenvalues, then the QCs~\eqref{eq:Rank3ValleyQC_phi} with $Q = -\tilde{Q}$ are valid QCs with identical proof. Also, Theorem~\ref{theorem:Rank3QC} recovers Theorem~\ref{theorem:ValleyQC} if we choose $c_3=0$ in~\eqref{eq:Rank3ValleyQC_phi} for the special case when $Q$ being rank 2. 

%% file: Methods/TOSO_QC.tex
\subsection{QC on the Cross-product of Monomials}     \label{sec:CrossQC}



Section~\ref{sec:ValleyQC} and~\ref{sec:Rank3QC} consider QCs to bound the effect of a single quadratic function. This section considers QCs to bound the cross-product of two monomials. This generalizes the QC~\eqref{eq:TosoQC} developed previously in~\cite{toso2022regional}. \commentg{The next lemma provides an upper and lower bound on the cross-product.}
\begin{lemma}       \label{lemma:CrossP}
    Let $w_p$ and $w_q$ be quadratic monomials such that the cross-product has the form $w_pw_q = x_i^2x_jx_k$ with $j\neq k$. The following inequalities hold:
    \begin{align}
        -x_i^2(x_j-x_k)^2 \leq 2w_pw_q \leq x_i^2(x_j+x_k)^2. \label{eq:CrossP_relaxation}
    \end{align}
\end{lemma}
\begin{proof}
    Note that $2x_i^2x_jx_k = x_i^2 ((x_j+x_k)^2 - x_j^2 - x_k^2)$ and hence $2w_pw_q \leq x_i^2(x_j+x_k)^2$. Similarly, $2w_pw_q = x_i^2 (x_j^2 + x_k^2 - (x_j-x_k)^2)$ and hence $2w_pw_q \geq -x_i^2(x_j-x_k)^2$.
\end{proof}

The next theorem utilizes the bounds~\eqref{eq:CrossP_relaxation} to provide QCs on the cross-product of monomials with the method developed in Section~\ref{sec:ValleyQC}.
\begin{theorem}[Cross-Product QC]     \label{theorem:CrossQC}
    Let $w_p$ and $w_q$ be quadratic monomials such that their cross-product has the form $w_pw_q = x_i^2x_jx_k$ with $j\neq k$. Then the cross-product satisfies the following four QCs:
    \begin{subequations}    \label{eq:CrossQC}
    \begin{align}
        \LMIVarXW^\top
        &\bmtx \alpha^2 W & 0 \\ 0 & \pm S_{pq} \emtx 
        \LMIVarXW \geq 0 \quad \forall x \in \mathcal{E}_\alpha \label{eq:CrossQC_QC}\\
        W &= (e_i^\top E^{-1}e_i)dd^\top \text{ or } (d^\top E^{-1} d) e_i e_i^\top, 
     \end{align}
    \end{subequations}
    where $d = e_j \mp e_k$, $S_{pq} = \bar{e}_p\bar{e}_q^\top + \bar{e}_q\bar{e}_p^\top$, and $e_i,e_j,e_k\in\mathbb{R}^n, \bar{e}_p,\bar{e}_q \in\mathbb{R}^m$ are standard basis vectors.
\end{theorem}
\begin{proof}
    Note that $2w_pw_q \leq x_i^2(x_j+x_k)^2$ by Lemma~\ref{lemma:CrossP}. Furthermore, $x_i^2(x_j+x_k)^2\leq x^\top (\alpha^2 W)x$ for the case $d=e_j+e_k$ for all $x\in\Ealpha$ by Lemma~\ref{lemma:Valley} with $c_1 = e_i$ and $c_2 = d$. Substitute $2w_pw_q = w^\top S_{pq} w$ and re-arrange the inequality into the quadratic form to obtain~\eqref{eq:CrossQC_QC} with $-S_{pq}$. The case $+S_{pq}$ and $d=e_j-e_k$ is shown similarly with inequality $-x_i^2(x_j-x_k)^2 \leq 2w_pw_q$ from Lemma~\ref{lemma:CrossP}. 
\end{proof}

The index $i$ can be arbitrary in Lemma~\ref{lemma:CrossP} and Theorem~\ref{theorem:CrossQC}, including $j$ or $k$. Note that the case $j=k$ is excluded, as the relaxation~\eqref{eq:CrossP_relaxation} is not necessary. The corresponding QCs for the case $w_pw_q = x_i^2x_j^2$ can be formed as:
\begin{align}
        \LMIVarXW^\top
        &\bmtx \alpha^2 W & 0 \\ 0 & -\frac{1}{2}S_{pq} \emtx 
        \LMIVarXW \geq 0 \quad \forall x \in \mathcal{E}_\alpha \label{eq:CrossQC_xixj}
\end{align}
with $W = (e_i^\top E^{-1}e_i)e_je_j^\top \text{ or } (e_j^\top E^{-1} e_j) e_i e_i^\top$.

The above bounding procedure in Lemma~\ref{lemma:CrossP} and Theorem~\ref{theorem:CrossQC} can be extended to a cross-product of general quadratic functions with similar structures. However, more specific conditions and the resulting QCs require dedicated study for the particular quadratic functions of interest.

%% file: Content/4_Results.tex
\section{Numerical Examples}


The proposed QCs are illustrated via an ROA estimation problem. The analysis algorithm is adopted from~\cite{kalur2021LCSS} with simplification detailed in the following paragraph. Note that the proposed QC can also be incorporated into any algorithm utilizes QC, e.g.~\cite{toso2022regional}. The intention of this section is to compare the effectiveness of newly introduced QCs to the existing QCs without involing advanced algorithms for this particular analysis. Hence, comparison against the full algorithms in~\cite{kalur2021LCSS} and~\cite{toso2022regional} is not provided. 

The largest ROA estimation is obtained by maximizing $r$ over $P,r,\xi$ subject to constraints~\eqref{eq:ROA} in Theorem~\ref{theorem:ROA}. This optimization problem is an SDP for given $E$ and $\alpha$ and hence the optimal $r^*$ can be solved efficiently. The largest $r^*$ is computed over a grid of $\alpha$ with a given shape local region $E$. In this paper, $E=I$ is fixed in all comparisons. The results can be improved by iteratively updating $E$ as in~\cite{kalur2021LCSS}.

The existing CSQC~\eqref{eq:CSQC} from literature~\cite{kalur2021LCSS} serves as the baseline analysis. It is compared against results that incorporate the proposed Rank-2 Valley QCs~\eqref{eq:ValleyQC}, Rank-3 Valley QCs~\eqref{eq:Rank3ValleyQC_phi}, and Cross-Product QCs~\eqref{eq:CrossQC}. 

A 2-state system and a 3-state system are investigated. Both examples were implemented in MATLAB with CVX~\cite{cvx} and the SDP solver MOSEK~\cite{mosek}. 
The implementation is made available online\footnote{
Source code is available at \url{https://github.com/SCLiao47/ValleyQC\_ROA} titled \textit{ValleyQC\_ROA} on \textit{GitHub.com}.}. Note that the effectiveness of proposed work depends on the specific dynamics. For example, the four-state shear flow problem discussed in~\cite{Liu2020io-inspired,kalur2021LCSS,toso2022regional} is not included as the new QCs provide only small improvement on this example. 

\subsection{2-state Example}

Consider the quadratic nonlinear system~\cite{amato2006region}:
\begin{align}   \label{Example:2State}
    \frac{d}{dt}\bmtx x_1 \\ x_2 \emtx = \bmtx -50 & -16 \\ 13 & -9 \emtx \bmtx x_1 \\ x_2 \emtx + \bmtx 13.8 \\ 5.5 \emtx x_1x_2.
\end{align}
The system has one stable equilibrium at the origin. By simulating trajectories, the phase portrait (Fig.~\ref{fig:example_2State}) indicates the largest spherical ROA has a radius about $r^*\approx4.95$ with the unstable region at the upper-right.

\begin{figure}
    \centering
    \vspace*{0.2cm}
    \includegraphics[width=\linewidth]{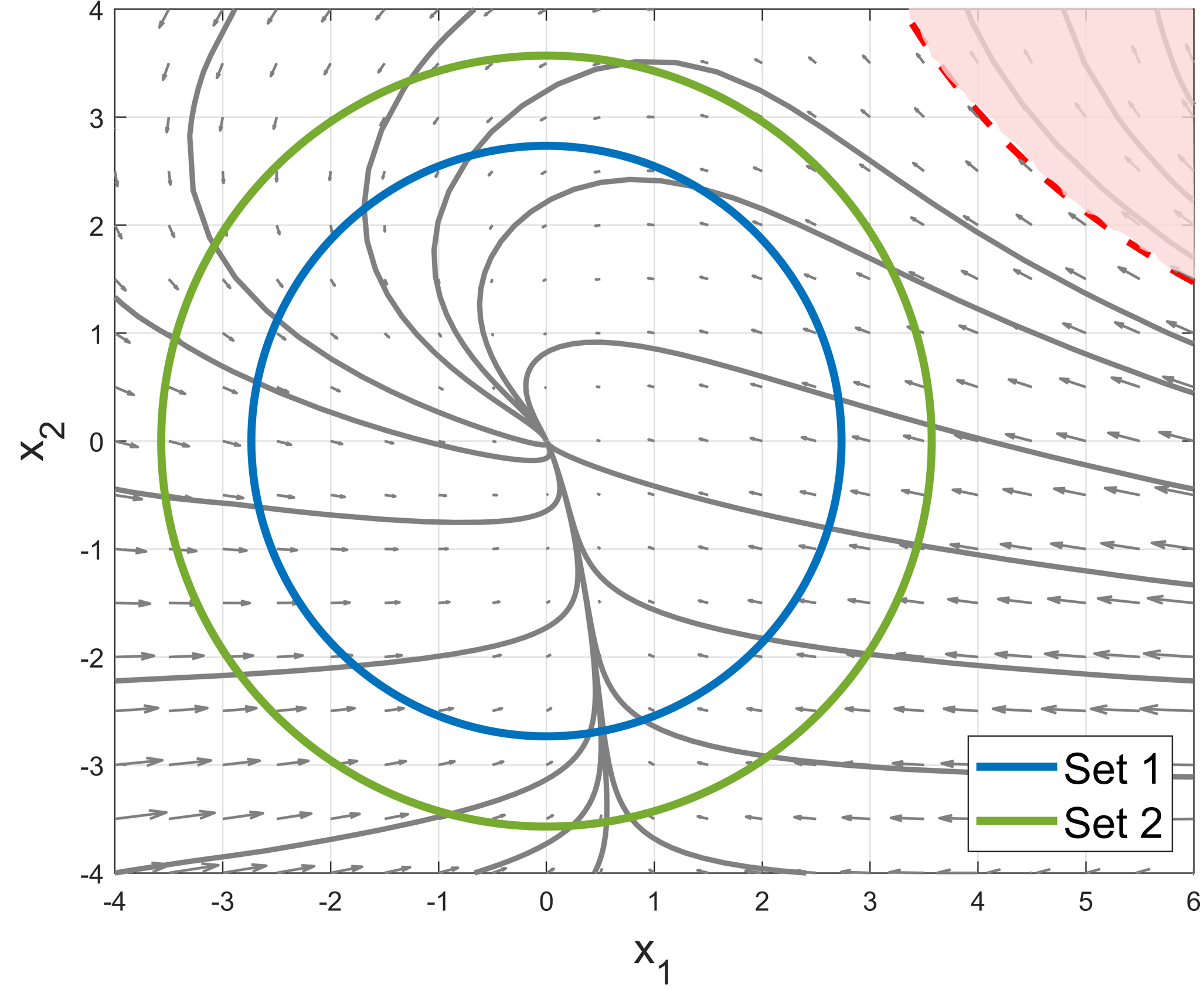}
    \caption{Phase portrait and ROA analysis results for the 2-state system in~\eqref{Example:2State}. The upper-right region shaded in red is an unstable region. The two circles are the ROA estimates given by each QC set. Set 1 applies CSQC~\eqref{eq:CSQC}, and Set 2 applies CSQC~\eqref{eq:CSQC} and Rank-2 Valley QCs~\eqref{eq:ValleyQC}.}
    \label{fig:example_2State}
\end{figure}

The nonlinearity $x_1x_2$ in system~\eqref{Example:2State} can be bounded by both the CSQC~\eqref{eq:CSQC} and Rank-2 Valley QCs~\eqref{eq:ValleyQC}. The ROA estimation is performed with two sets of QCs. Set 1 applies CSQC~\eqref{eq:CSQC} on the nonlinearity $x_1x_2$. Set 2 applies both CSQC~\eqref{eq:CSQC} and Rank-2 Valley QCs~\eqref{eq:ValleyQC} on $x_1x_2$.

The QC analysis results are visualized in Fig.~\ref{fig:example_2State}. 
Set~2 ($r_2^*=3.5224$) gives a less conservative estimate than Set~1 ($r_1^*=2.7355$) by incorporating the Rank-2 Valley QCs~\eqref{eq:ValleyQC}. The results illustrate the effectiveness of Rank-2 Valley QCs~\eqref{eq:ValleyQC} on the quadratic nonlinearity in~\eqref{Example:2State}.



\subsection{3-state Example}
Consider the 3-state system $\dot{x} = Ax + Bw$:
\begin{align}   \label{Example:3state}
    \begin{split}
        A &= \bsmtx -1 & -1 & -1 \\ -1 & -6 & -1 \\ -1 & -1 & -13\esmtx, \quad
        B = \bsmtx 0 & 1 & 0 & 0 & 0 & 1 \\ 0 & 0 & -6 & 0 & -4 & 0 \\ 0 & 0 & 0 & 0 & 2 & -1 \esmtx, \\
        x &= \bsmtx x_1 & x_2 & x_3 \esmtx^\top, \quad 
        w = \bsmtx x_1^2 & x_1x_2 & x_1x_3 & x_2^2 & x_2x_3 & x_3^2 \esmtx^\top.    
    \end{split}
\end{align}
\commentg{By numerically solving trajectories, the spherical ROA estimate of the system has an upper bound $r^*\approx2.4283$, where there exists an initial condition not converging to $x_e = 0$.} The largest spherical ROA estimate has a radius smaller than $\bar{r}$.

The analysis is performed with different sets of QCs. The CSQC~\eqref{eq:CSQC} is applied to each monomial and $b_i^\top w$ for $i=1,2,3$, where $b_i^\top$ is the $i$-th row vector of the matrix $B$. The Rank-2 Valley QCs~\eqref{eq:ValleyQC} are applied to each sign-indefinite monomial $w_i$, $b_2^\top w$ and $b_3^\top w$. The Rank-3 Valley QCs~\eqref{eq:Rank3ValleyQC_phi} are applied on $b_1^\top w$. The Cross-Product QCs~\eqref{eq:CrossQC} are applied on each pair of monomials satisfying the conditions.

TABLE~\ref{tab:example_3state} summarizes the setting of analysis and the results for eight sets of QCs. Each of the Set 2, 3, and~4 gives a less conservative result than Set 1. The results indicate that each of the proposed QCs improved the analysis individually. Furthermore, Set 5, 6, and 7 show that the analysis result could be improved by including multiple proposed QC into the analysis. Lastly, Set 5 and Set 8 give the least conservative estimation among all sets. The two analysis are the same up to the numerical tolerance of the solver. While this might imply adding Cross-Product QCs does not improve the analysis, this could be because of the specific system~\eqref{Example:3state} and stability condition used. Another system or stability condition could have different results. 

\begin{table}[h!]
\centering
\caption{ROA estimation results $r^*$ of 3-state system in~\eqref{Example:3state} by QC sets.}
\begin{tabular}{r|cccccc}
\textbf{Set \#} & CSQC          & Rank-2        & Rank-3        & Cross-Product & QC \# & $r^*$  \\ \hline
\textbf{Set 1}  & \checkmark    &               &               &               & 9     & 0.7173  \\ 
\textbf{Set 2}  & \checkmark    & \checkmark    &               &               & 19    & 1.2041 \\ 
\textbf{Set 3}  & \checkmark    &               & \checkmark    &               & 13    & 0.8487 \\ 
\textbf{Set 4}  & \checkmark    &               &               & \checkmark    & 63    & 0.7900 \\ 
\textbf{Set 5}  & \checkmark    & \checkmark    & \checkmark    &               & 23    & 1.3365 \\ 
\textbf{Set 6}  & \checkmark    & \checkmark    &               & \checkmark    & 73    & 1.2468 \\ 
\textbf{Set 7}  & \checkmark    &               & \checkmark    & \checkmark    & 67    & 0.8846 \\ 
\textbf{Set 8}  & \checkmark    & \checkmark    & \checkmark    & \checkmark    & 77    & 1.3365 \\ 
\end{tabular}
\label{tab:example_3state}
\end{table}

%% file: Content/5_Conclusions.tex
\section{Conclusions}


In this work, we proposed new quadratic constraints to reduce conservatism in the analysis of quadratic systems using dissipation inequalities. The proposed QCs exploit the property of sign-indefinite quadratic polynomials to tighten the bound along with the QC previously derived in~\cite{kalur2021LCSS}. The effectiveness of the proposed QCs is illustrated by successfully enlarging ROA estimations in two numerical examples. Future work includes applying the QCs to other system analysis problems and investigating the computational scalability of the proposed method.

%% file: Content/Appendix.tex
\appendices

\section{Rank 2 Sign-indefinite $Q$}    \label{apdx:Rank2}

This appendix shows that a matrix $Q=Q^\top\in\Ri{n\times n}$ is rank 2 with one positive and one negative eigenvalue if and only if exists nonzero, linearly independent vectors $c_1, c_2 \in\Rn$ such that $Q=\frac{1}{2}(c_1c_2^\top+c_2c_1^\top)$ with.

Let $(\lambda_p, v_p)$ and $(\lambda_n, v_n)$ be the eigenpairs of $Q$ associated with the positive and negative eigenvalues. $Q$ has the eigendecomposition $Q=\lambda_pv_pv_p^\top + \lambda_nv_nv_n^\top$. For any $x\in\Rn$, the following is true:
\begin{align*}
    x^\top Q x &= (\sqrt{\lambda_p} v_p^\top x)^2 - (\sqrt{|\lambda_n|}v_n^\top x)^2 \\
    &= (c_1^\top x)(c_2^\top x) = \frac{1}{2}x^\top (c_1c_2^\top+c_2c_1^\top) x
\end{align*} 
with $c_1 = \sqrt{\lambda_p}v_p + \sqrt{|\lambda_n|}v_n$ and $c_2 = \sqrt{\lambda_p}v_p - \sqrt{|\lambda_n|}v_n$. Note that $v_p, v_n$ is orthogonal to each other since they are eigenvectors of the real symmetric matrix $Q$. Hence, $c_1, c_2$ are nonzero, linearly independent vectors as $\lambda_p,\lambda_n\neq 0$.

Conversely, let $c_1,c_2\in\Rn$ be any nonzero, linearly independent vectors. Define $[c_1, c_2] = UR$ be the QR factorization, where $U\in\Ri{n\times n}$ is an orthogonal matrix and $R\in\Ri{n\times2}$. Define $Q = \frac{1}{2}(c_1c_2^\top+c_2c_1^\top)$. 
\commentg{One can verify that
\begin{align*}
    U^\top Q U = 
    \left[\begin{array}{c|c}
        \begin{smallmatrix}
         a & b \\ b & 0
        \end{smallmatrix} &  0 \\ \hline
        0 & 0
    \end{array}\right]
\end{align*}
has a nonzero block with some $a\in\Ri{}$ and $b\neq 0$.} The eigenvalues of $Q$ are $\frac{a\pm\sqrt{a^2+4b^2}}{2}$ and $(n-2)$ repeated eigenvalues at $0$. Note that $\sqrt{a^2+4b^2} > a$. Hence, $Q$ is rank 2 with one positive and one negative eigenvalue.

\section{Rank 3 Sign-indefinite $Q$}    \label{apdx:Rank3}

This appendix shows that a matrix $Q=Q^\top\in\Ri{n\times n}$ is rank 3 with two positive and one negative eigenvalue if and only if exists nonzero, linearly independent vectors $c_1, c_2, c_3 \in\Rn$ such that $Q=\frac{1}{2}(c_1c_2^\top+c_2c_1^\top)+c_3c_3^\top$ with $c_3$ being orthogonal to $c_1,c_2$.

Let $(\lambda_i, v_i)$ be the eigenpairs of $Q$ for $i=1,2,3$, where $\lambda_1,\lambda_2$ are positive and $\lambda_3$ is negative. For any $x\in\Rn$, the following is true from the eigendecomposition of $Q$:
\begin{align*}
    x^\top Q x =  (\sqrt{\lambda_1} v_1^\top x)^2 + (\sqrt{\lambda_2}v_2^\top x)^2 - (\sqrt{|\lambda_3|} v_3^\top x)^2.
\end{align*}
Hence, $Q$ can be written as $\frac{1}{2}(c_1c_2^\top+c_2c_1^\top)+c_3c_3^\top$ with vectors $c_1 = \sqrt{\lambda_1}v_1 + \sqrt{|\lambda_3|}v_3, c_2 = \sqrt{\lambda_1}v_1 - \sqrt{|\lambda_3|}v_3$ and $c_3 = \sqrt{\lambda_2}v_2$. Note that $c_1,c_2,c_3$ are nonzero, linearly independent due to the eigenvectors of real symmetric matrix $Q$ being orthogonal. Also, $c_3$ is orthogonal to $c_1,c_2$. 

Conversely, let $c_1,c_2,c_3$ be nonzero, linearly independent vectors with $c_3$ being orthogonal to $c_1,c_2$. Observe that $\norm{c_3}>0$ is a positive eigenvalue of $Q=\frac{1}{2}(c_1c_2^\top+c_2c_1^\top)+c_3c_3^\top$ with associated eigenvector $c_3$. It follows from Appendix~\ref{apdx:Rank2} that the space spanned by $(c_1,c_2)$ contains one positive and one negative eigenvalue. Hence, $Q$ is rank 3 with two positive and one negative eigenvalue. 